\documentclass[10pt,conference, english]{IEEEtran}
                                                        
\usepackage[inner=1in,outer=1in,top=1in,bottom=1in, marginparwidth=.7in,marginparsep=.1in]{geometry}

\usepackage[T1]{fontenc}
\usepackage{amssymb}

\usepackage{colordvi}
\usepackage{graphicx}

\usepackage{paralist}
\usepackage{babel}
\usepackage{amsmath}
\usepackage{amsthm}
\usepackage{tikz}
\usepackage{url,  mathrsfs}
\usepackage{hyperref}

\newcommand\C{\mathbb{C}}

\newcommand\Q{\mathbb{Q}}
\newcommand\R{\mathbb{R}}

\DeclareMathOperator\trace{trace}

\theoremstyle{plain}
\newtheorem{Thm}{Theorem}

\newtheorem*{Thm*}{Theorem}
\newtheorem*{Prop*}{Proposition}
\newtheorem*{Cor*}{Corollary}
\newtheorem*{Lemma*}{Lemma}
\newtheorem*{Conjecture*}{Conjecture}

\theoremstyle{definition}

\newtheorem{Example}[Thm]{Example}

\newtheorem*{Constr*}{Construction}
\newtheorem*{Def*}{Definition}
\newtheorem*{Example*}{Example}
\newtheorem*{Remark*}{Remark}

\begin{document}
\title{A small frame and a certificate of its injectivity}

\author{
\IEEEauthorblockN{Cynthia Vinzant}
\IEEEauthorblockA{
Department of Mathematics\\
North Carolina State University \\
Raleigh, North Carolina 27695, USA\\
Email: clvinzan@ncsu.edu}}

\maketitle

{\abstract 
We present a  complex frame of eleven vectors in 4-space and prove that it defines injective measurements. 
That is, any rank-one $4\times 4$ Hermitian matrix is uniquely determined by its values as a Hermitian form on this collection of eleven vectors. 
This disproves a recent conjecture of Bandeira, Cahill, Mixon, and Nelson. We use algebraic computations and certificates 
in order to prove injectivity. 
}

A finite-dimensional \emph{frame} is a collection of vectors $ \Phi = (\phi_1, \hdots, \phi_n)$ that span 
$\C^d$ and  define measurements 
\[|\langle\phi_k, x \rangle|^2 \ = \ \trace(\phi_k^{\;}\phi_k^*xx^*) \ \in \ \R_{\geq 0}\]
 on a signal $x\in \C^d$. 
The problem of \emph{phase retrieval} is to recover the rank-one matrix $xx^*$ from these measurements.  
This relates to low-rank matrix completion and has many imaging-related applications: microscopy, optics, 
and diffraction imaging, among others. 

Phase retrieval is an old problem in signal processing and there are
many interesting and important questions on the topic, including 
how one recovers the matrix $xx^*$ and the stability of this recovery \cite{CESV:13}.
This paper focusses on the question of when that recovery is possible, \emph{i.e.}
when the measurements $ \trace(\phi_k^{\;}\phi_k^*xx^*)$ uniquely determine
the rank-one matrix $xx^*$. 

We say that a frame $\Phi$ \textbf{defines injective measurements} on $\C^d$ 
if the linear map from $d\times d$ Hermitian matrices to $\R^n$ defined by 
$Q \mapsto (\phi_1^*Q\phi_1^{\;}, \hdots, \phi_n^*Q\phi_n^{\;})$
is injective on the set of rank-one Hermitian matrices. That is, $\Phi$ defines injective measurements 
if $xx^* = yy^*$ for any two vectors $x,y \in \C^d$ that have equal measurements 
$\phi_k^*xx^*\phi_k^{\;}= \phi_k^*yy^*\phi_k^{\;}$ for all $k=1, \hdots, n$.

In the context of finite frame theory, the problem of injective measurements was first studied by
Balan, Casazza, and Edidin \cite{BCE:06}, who show when $n\geq 4d-2$ a generic frame defines injective measurements.
In the other direction, 
Heinosaari, Mazzarella and Wolf prove that $n\geq(4 + o(1))d$ is necessary for $\Phi$ to define injective measurements. 
Specifically  \cite[Theorem~6]{HMW:13}, they use embedding theorems in homotopy theory to show that one needs 
$n >  4d-4-2\alpha$ where $\alpha$ is the number of $1$'s in the 
binary expansion of $d-1$. 

Bandeira, Cahill, Mixon, and Nelson  \cite{BCMN:13} conjectured that 
fewer than $4d-4$ measurements cannot be injective on $\C^d$, whereas $4d-4$ generic
measurements are injective. Recently, the latter part of this was proved by  
Conca, Edidin, Hering, and the current author \cite{CEHV:14}, who show that 
when $n\geq 4d-4$ a generic frame $\Phi \in \C^{d\times n}$ defines injective measurements on $\C^d$.
They also show that in the case $d=2^k+1$ and $n < 4d-4$, any frame does not define injective measurements.

In this paper, we present a counterexample to the first part of this conjecture in the smallest open case, $d=4$. 
In Section~\ref{sec:ex}, we give a frame consisting of $11 = 4d-5$ vectors in $\C^4$ and 
prove that it defines injective measurements on $\C^4$.  
The example was found via a random search and the proof is computational, using certificates in algebraic 
and real algebraic geometry. 
We hope that this will spur the search for more systematic counterexamples that may extend to $d>4$. 

%
%

\section{Translation to polynomials}\label{sec:translate}

Following the set up of \cite{CEHV:14}, we translate the injectivity of measurements into 
a condition on the solutions of a system of polynomial equations. 

A useful step 
is the following reformulation of injectivity by Bandeira \textit{et. al.} \cite[Lemma 9]{BCMN:13}. 
They observe that a frame $\Phi$ defines 
injective measurements if and only if the linear space 
\[ \mathcal{L}_{\Phi}  =  \{ Q \in \C^{d\times d} \; :\; \phi_1^*Q\phi_1^{\;}=\hdots= \phi_n^*Q\phi_n^{\;} =0\}  \]
does not contain any non-zero Hermitian matrices of rank $\leq 2$. 
The existence of rank $\leq 2$ Hermitian matrices in $\mathcal{L}_{\Phi}$ can be 
rephrased as the existence of real roots of a certain system of polynomial equations as follows.

Any $4\times 4$ Hermitian matrix can be written as $Q =$
\[
\begin{pmatrix}
x_{11} & x_{12} + i y_{12} & x_{13} + i y_{13} & x_{14} + i y_{14} \\
x_{12} - i y_{12} & x_{22} & x_{23}+i y_{23} & x_{24}+i y_{24} \\
x_{13} - i y_{13} & x_{23} - i y_{23} & x_{33} & x_{34}+i y_{34} \\
x_{14} - i y_{14} & x_{24} - i y_{24} & x_{34} - i y_{34} & x_{44}
\end{pmatrix},
\]
where $x_{11}, \hdots, y_{34}$ are real numbers. We will write our polynomial condition for injectivity
in the 16 variables $x_{jk}, y_{jk}$.   
For $1\leq  j,k \leq 4$, let $m_{jk}$ denote the determinant of the $3\times 3$ 
matrix obtained by removing the $j$th row and $k$th column from the matrix $Q$. 
The matrix $Q$ has rank $\leq 2$ when all these minors $m_{jk}$ equal zero.

Given a frame $\Phi = (\phi_1, \hdots, \phi_n) \in \C^{4\times n}$, define the real linear forms 
 $\ell_k = \phi_k^*Q\phi_k^{\;}$ in the variables $x_{jk}, y_{jk}$ for $ 1\leq j \leq k \leq 4$.  
Rephrasing  \cite[Lemma 9]{BCMN:13}, a frame $\Phi$ defines injective measurements on $\C^4$ if 
and only if the system of equations
\begin{equation}\label{eq:eqs}
m_{11} = m_{12} = \hdots = m_{44} =  \ell_1 = \hdots = \ell_{n} = 0  
\end{equation}
has \emph{no} non-zero real solution $(x_{jk}, y_{jk} ) \in \R^{16}$.
 
Note that this system of equations may have non-real complex solutions $(x_{jk}, y_{jk}) \in \C^{16}$, which 
correspond to non-Hermitian matrices in $\C^{4\times 4}$. 
 The set of matrices of rank $\leq 2$ is a homogeneous variety of dimension 12 and degree 20 
 inside of $\C^{4\times 4}$ \cite[Prop.~12.2, Ex. 19.10]{Har:95}. This means that a generic 
 linear space of codimension 11 intersects this variety in 20 complex lines (all passing through the origin). 
 This is the case for the linear space $\mathcal{L}_{\Phi}$ defined by the frame $\Phi$ below, and 
 the twenty lines come in ten complex conjugate pairs.

 \section{A small injective frame} \label{sec:ex}
 
The frame $\Phi= (\phi_1, \hdots, \phi_{11})$, 
consisting of the rows of the matrix
\begin{equation}\label{eq:smallFrame}
\Phi^T = 
\left(
\begin{array}{cccc}
 1 & 0 & 0 & 0 \\
 0 & 1 & 0 & 0 \\
 0 & 0 & 1 & 0 \\
 0 & 0 & 0 & 1 \\
 1 & 9 i & -5-7 i & -6-7 i \\
 1 & 1-i & -5-2 i & -1-8 i \\
 1 & -2+4 i & -4-2 i & 3+8 i \\
 1 & -3+i & 1-8 i & 7-6 i \\
 1 & 3-3 i & -8+7 i & -6-2 i \\
 1 & -3+5 i & 5+6 i & 2 i \\
 1 & -3+8 i & 5-5 i & -6-4 i \\
\end{array}
\right),
\end{equation}
defines injective measurements on $\C^4$, and therefore provides a counterexample to Conjecture~2 of \cite{BCMN:13}. 
That is, for any $x\in \C^4$, the values $|\langle\phi_k, x \rangle|^2$ for $k=1, \hdots, 11$ uniquely determine the rank-one matrix $xx^*$.

Apart from the coordinate vectors, the vectors of $\Phi$ were chosen to have first coordinate 1 and otherwise found 
 by a random search. This matrix is by no means unique, as further discussed in Section~\ref{sec:semialg}.

\begin{Thm} \label{thm:main}
$\Phi$ in \eqref{eq:smallFrame} defines injective measurements. 
\end{Thm}
\begin{proof}

Using the set-up of Section~\ref{sec:translate}, it suffices to show the equations \eqref{eq:eqs}
have no non-zero real solution.

Using computer algebra software such as {\tt Macaulay2} \cite{M2} or {\tt Mathematica} \cite{Mathematica}, 
one can compute a Gr\"obner basis of the set of polynomials $\{m_{11}, \hdots, m_{44}, \ell_1, \hdots, \ell_{11}\}$ 
and eliminate all of the variables except two from the system of equations~\eqref{eq:eqs}. 
See \cite[Ch. 3]{IVAbook} for background on Gr\"obner bases and elimination. 
The result is a polynomial $f(x_{34},y_{34})\in \Q[x_{34},y_{34}]$  (shown in Figure~1) with the property that $f(x_{34},y_{34})=0$
if and only if the point $(x_{34},y_{34})\in \C^2$ can be extended to a solution 
$(x_{11} \hdots, x_{44}, y_{12}, \hdots, y_{34})\in \C^{16}$
of the system of equations \eqref{eq:eqs}.  
This is the minimal polynomial in $\Q[x_{34},y_{34}]$ that can be written as 
\[f(x_{34},y_{34})  \ \ = \ \  \sum_{j,k} p_{jk}m_{jk} + \sum_k q_k \ell_k \]
with polynomials $p_{jk} , q_k \in \Q[i][x_{jk}, y_{jk} : 1\leq j, k\leq 4]$.

\begin{figure*} \label{fig:poly}
\begin{center}
{\scriptsize
\begin{tabular}{l}
$f(x,y)=$ \vspace{.05cm}\\
$47599685697454466246329412358483179722150043354437125082025800902606928597206272254845887202098485215232\cdot x^{20}$ \vspace{.05cm} \\
$-940875789867758769838520754403201268675774719194241940388656177785644194342166892793123967870118511091712\cdot x^{19}y$ \vspace{.05cm}\\
$+8079760677210192071804090111142610477024725441627364213141746522285905327070793719538623768982021441867008\cdot x^{18}y^2$ \vspace{.05cm}\\
$-40390761193855122277381198616744763479497680895608897593386520810794749041801633796968256299345250567989120\cdot x^{17}y^3$ \vspace{.05cm}\\
$+131616369916171208334977339064503371859576391268929064468118935900017365295185627042078382592920359023963120\cdot x^{16}y^4$\vspace{.05cm} \\
$-293014395329583025877260372789628942263338515685834588963896339613217690953560112063134591204469166903730584\cdot x^{15}y^5$\vspace{.05cm} \\
$+458069738032730695996144135248791338007569710877529938378092745783077549558976157025550745961972225340079644\cdot x^{14}y^6$\vspace{.05cm} \\
$-517369071593627219847520943924454458561147451524495675098907021370976281217299640311489465704692368615264514\cdot x^{13}y^7$ \vspace{.05cm}\\
$+452598979230255288442671627934707378002747893014717388494818021654528875197345624154508626114037972901500688\cdot x^{12}y^8$\vspace{.05cm} \\
$-372648962908998912506284086331829334659704158038572388762607081397540397891875288020327841800275807896331363\cdot x^{11}y^9$\vspace{.05cm} \\
$+368232864821580663608362507224731842224816948166375792251958189898413349943059199991850745920857587346422247\cdot x^{10}y^{10}$ \vspace{.05cm}\\
$-403635711731885683831862286003879871368285836090576953930238823174701111263082513174328319091824845878408842\cdot x^9y^{11}$ \vspace{.05cm}\\
$+390921191544945060106454097348764080175218877410156079207976994796588444804574583525852046116133406063492232\cdot x^8y^{12}$\vspace{.05cm} \\
$-303282246743535677380017745889681371136540419380112690433239947491979764226862379182777142974211242201436038\cdot x^7y^{13}$ \vspace{.05cm}\\
$+184479380320049045197686505443823960153384609428987780432573005109397657926440688558298683493092343685387706\cdot x^6y^{14}$ \vspace{.05cm}\\
$-87485311349460982824448992498046043498427396179321650198242819939653352363165057564278033789500273373973662\cdot x^5y^{15}$\vspace{.05cm} \\ 
$+32016520763724676437134174594818955536984857769461915546273804322365856693290090903851788729777275040411744\cdot x^4y^{16}$\vspace{.05cm} \\ 
$-8843043103455739360596137302837349740785483274132912552686735695145524362028265118639059872092039716064999\cdot x^3y^{17}$ \vspace{.05cm}\\ 
$+1775125426181341100587099980276312627299716879819457817398603067248151810981307579223879621865024794510283\cdot x^2y^{18}$ \vspace{.05cm}\\ 
$-241527118652311488433038772168913074025991214453188628647589057033246072076996489577531666185336332308462\cdot x y^{19}$ \vspace{.05cm}\\ 
$+17892217832720483440399845902831090202434763229104212220658085110841220106091148070445766234106381722000 \cdot y^{20}$
   \end{tabular}}
\caption{The minimal polynomial $f$ that vanishes at all points $(x,y) \in \C^2$ for which there exists a rank-2 matrix $Q$ in 
$\mathcal{L}_{\Phi}$ with $Q_{3,4} = x+iy$ and $Q_{4,3} = x-iy$.  As discussed in the proof of Theorem~\ref{thm:main}, the only solution to 
$f(x,y)=0$ in $\R^2$ is $(x,y)=(0,0)$.}
\end{center} 
\end{figure*}

Such a certificate verifies that $f(x_{34},y_{34}) =0 $ for all solutions 
to \eqref{eq:eqs}. Unfortunately the polynomial multipliers $p_{jk}$, $q_k$ involved
are too large to reproduce here. 

We will examine solutions to \eqref{eq:eqs} via the solutions of $f(x_{34},y_{34})=0$.
Like the original system \eqref{eq:eqs}, the solution set to 
$f(x_{34},y_{34})=0$ is invariant under scaling. That is, $f(x_{34},y_{34})=0$ implies that 
$f(\lambda x_{34},\lambda y_{34})=0$ for all scalars $\lambda\in \C$. 
In particular, $f(x_{34},y_{34})=0$ has a \emph{real} solution $(x_{34},y_{34})\in \R^2$ with $y_{34}\neq 0$ if and only if 
it has a real solution with $y_{34}=1$. 
However, using Sturm sequences \cite[\S 2.2.2]{BPRbook}, \cite{FrankM2}, one can verify that the univariate polynomial 
$f(x_{34}, 1) \in \Q[x_{34}]$ has no real roots.  This shows that the system of equations
\eqref{eq:eqs} has no real solutions with $y_{34}\neq 0$. 

One can also verify that there is no non-zero solution to $\eqref{eq:eqs}$ in $\C^{16}$ with $y_{34}=0$ as follows.
Because the solution set of $\eqref{eq:eqs}$ is invariant under scaling, there is a non-zero solution to \eqref{eq:eqs}
if and only if there is a solution with some coordinate equal to 1. This can be checked one coordinate at a time. 
For example, computing a Gr\"obner basis of the set of polynomials
$\{x_{12}-1, y_{34}, m_{11}, \hdots, m_{44}, \ell_1, \hdots, \ell_{11}\}$ 
reveals that 
\[1  \  = \ r\cdot (x_{12}-1) + s\cdot y_{34} + \sum_{j,k} p_{jk}m_{jk} + \sum_k q_k \ell_k \]
for some $r,s, p_{jk} , q_k \in \Q[i][x_{jk}, y_{jk} : 1\leq j,k\leq 4]$. This certifies that 
there is no solution in $\C^{16}$ to \eqref{eq:eqs} with $x_{12}=1$ and $y_{34}=0$. 
Repeating this process with the other variables in place of $x_{12}$, one can certify that 
there is no non-zero solution to $\eqref{eq:eqs}$ in $\C^{16}$ with $y_{34}=0$. 

These computations complete the certification that there are no non-zero solutions 
to the system of equations \eqref{eq:eqs}. Thus there are no non-zero Hermitian matrices 
of rank $\leq 2$ in the linear space $\mathcal{L}_{\Phi}$ and, by  \cite[Lemma~9]{BCMN:13},
the frame $\Phi$ defines injective measurements. 

The code for these computations in both {\tt Macaulay2} and {\tt Mathematica} 
are available at \url{http://www4.ncsu.edu/~clvinzan/smallFrame.html}.
\end{proof}

Solving the system of equations \eqref{eq:eqs} numerically, we see that, up to scaling, there are 
exactly twenty rank-2 matrices in the linear space $\mathcal{L}_{\Phi}$. These are in 
one-to-one correspondence with the twenty complex roots of the polynomial $f(x_{34},1)$. 
As none of these roots are real, none of the rank-2 matrices in $\mathcal{L}_{\Phi}$ are Hermitian. 
For example, the solution $(x_{34},y_{34}) \approx (1.95 + 2.08 i, 1) $ corresponds to
the rank-2 matrix
\[{\scriptsize
\begin{pmatrix}
\! 0 \!\!\!& \!\!\!\!-12.84-22.02 i \!\!\!& \!\!-27.63-6.1 i \!\!\!& \!\!26.67 -31.13 i  \\
\! 30.12 -34.42 i \!\!\!& \!\!\!\! 0 \!\!\!& \!\!-3.48+4.16 i \!\!\!& \!\!\!1.24 -1.93 i   \\
\! 17.86 -16.81 i \!\!\!& \!\!\!\! 2.62 +0.13 i \!\!\!& \!\!0 \!\!\!& \!\!\!1.95 +3.08 i \\
 \! -15.06-15.68 i \!\!\!& \!\!\!\! 0.57 -0.37 i \!\!\!& \!\!1.95 +1.08 i \!\!\!& \!\!\!0 \\
\end{pmatrix}
}\]
and its conjugate, $(x_{34},y_{34}) \approx (1.95 - 2.08 i, 1) $, gives
\[{\scriptsize
\begin{pmatrix}
\! 0 \!\!\! & \!\! 30.12 +34.42 i \!\!\! & \!\!17.86 +16.81 i \!\!\! & \!\! -15.06+15.68 i \\
\! -12.84+22.02 i \!\!\! & \!\! 0 \!\!\! & \!\! 2.62 -0.13 i \!\!\! & \!\! 0.57 +0.37 i \\
\! -27.63+6.1 i \!\!\! & \!\! -3.48-4.16 i \!\!\! & \!\! 0 \!\!\! & \!\! 1.95 -1.08 i \\
\! 26.67 +31.13 i \!\!\! & \!\! 1.24 +1.93 i \!\!\! & \!\! 1.95 -3.08 i \!\!\! & \!\! 0 \\
 \end{pmatrix} \!\!.
}\]

Because the certificates used in the proof of Theorem~\ref{thm:main}
are too large to give here, we now present a much smaller 
example of the computations involved.

\begin{Example}\label{ex:tiny}
Suppose we want to show that there is no rank-one Hermitian matrix of the form
\[
\begin{pmatrix}
z & x+i y \\  x- i y & z
\end{pmatrix}
\]
satisfying the linear equation $\ell = x+y - 2z =0$.
This is equivalent to showing that there is no non-zero real solution 
to the equations $m=\ell=0$, where 
$m$ denotes the determinant $m=z^2 - x^2 - y^2$ of this matrix. 

While this problem could easily be solved by hand, 
we will follow the proof of Theorem~\ref{thm:main}. 
Computing a Gr\"obner basis of the polynomials $\{m, \ell\}$ 
to eliminate the variable $z$, we find the minimal 
polynomial, 
\[f(x,y) = 3x^2 - 2 x y + 3 y^2 = -4m - (x+y+2z)\ell,\]
in $\Q[x,y]$ that vanishes on all the points $(x,y,z)\in \C^3$ 
satisfying $m(x,y,z)=\ell(x,y,z)=0$.

\begin{figure} \begin{center}
\includegraphics[width=2.3in]{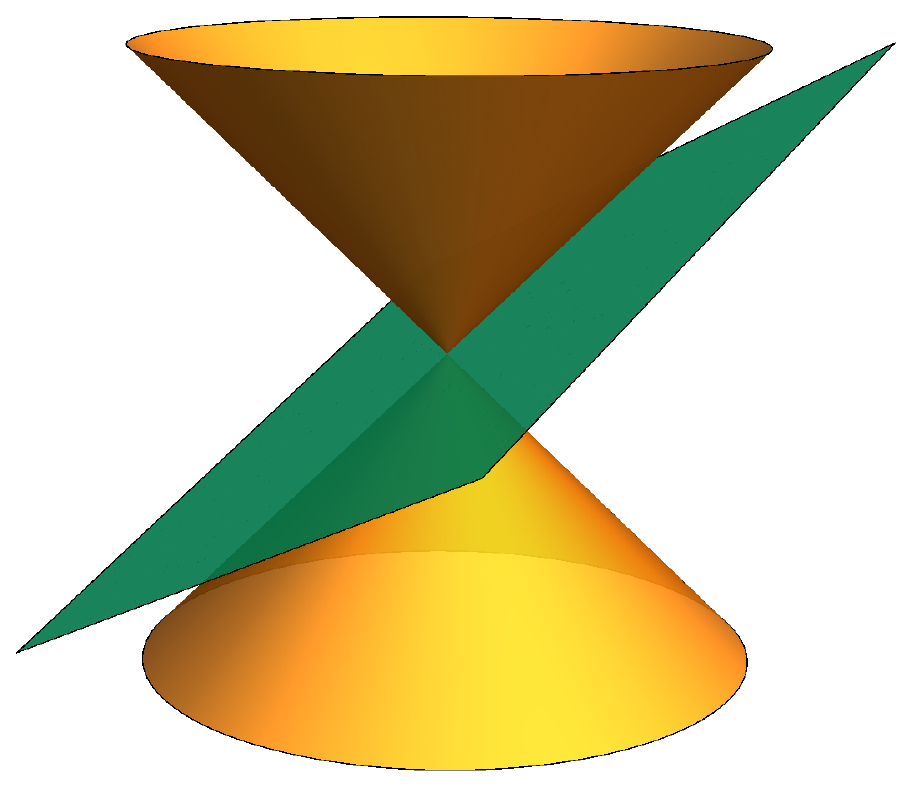}
\end{center}
\caption{The real solutions of $m=0$ and $\ell=0$  from Example~\ref{ex:tiny}.}
\end{figure}

The univariate polynomial $f(x,1) = 3x^2 - 2 x + 3$ has no real roots, 
which in this case can be verified directly. 
Because the solution set to $m=\ell=0$ is invariant under scaling, this 
implies that there is no real solution
of $m=\ell=0$ with $y\neq 0$. 

To check that there are no non-zero solutions with $y=0$, 
we first check for solutions with $x=1$ and then with $z=1$. 
Computing a Gr\"obner basis of the set $\{y,x-1,\ell, m\}$ reveals the expression 
\[1 =  \frac{(2x - 3y)}{3} y-(x+1) (x - 1)  -\frac{(x+y + 2 z)}{3} \ell - \frac{4 }{3}m .\]
This certifies that the there is no point $(x,y,z)$ satisfying $y=x-1=\ell=m=0$. 
(Plugging in such a solution to this equation would result in $1=0$.)
Next we compute a Gr\"obner basis of $\{y,z-1,\ell, m\}$ and see that
\[1  \ = \   \frac{2x}{3} y-(z+1) (z - 1)  - \frac{(x +y+ 2 z)}{3} \ell - \frac{1}{3}m,\]
which proves that the there is no solution to the equations $y=z-1=\ell=m=0$. 
Together these show that there is no non-zero solution of $\ell=m=y=0$. 

Hence the only real solution to $m=\ell=0$ is $(x,y,z)=(0,0,0)$. Indeed, we see in this small example that
 the solution set to $m=\ell=0$ is the union of two complex conjugate lines in $\C^3$, 
spanned by the rays $(2 \pm i\sqrt{2}, 2 \mp i \sqrt{2}, 2)$, whose only real point 
is the origin. 
\end{Example}

 \section{The set of injective frames} \label{sec:semialg}
 
 The example \eqref{eq:smallFrame} found above is not unique. 
 In fact, the set of injective frames is full dimensional in $\C^{4\times 11}$. 
 
As discussed in \cite[Remark~4.4]{CEHV:14}, the collection 
of $(U,V)\in \R^{4\times 11} \times \R^{4\times 11}$ for which the frame $U+iV$  
does \emph{not} define injective measurements is a closed semi-algebraic set. Its complement, the set of 
 $(U,V)$ for which the frame $U+iV$ does define injective measurements, 
 is therefore an open semi-algebraic set in $\R^{4\times 11} \times \R^{4\times 11}$.
 
 To see this, consider the system of equations \eqref{eq:eqs} with the entries of the
 frame $\Phi$ playing the role of parameters in the real linear forms
 $\ell_k = \phi^*_kQ\phi^{\;}_k$.  
  If, for a given $\Phi$ such as \eqref{eq:smallFrame},
 the system of equations \eqref{eq:eqs} has no non-zero real solutions, then 
 for any sufficiently small perturbation of $\Phi$, the perturbed system of equations 
 will also have no non-zero real solutions. Thus the set of injective frames 
 contains a small open ball around $\Phi$.
 
 For example, we can replace the last vector of $\Phi$, 
 \[{\phi_{11}} = \begin{pmatrix} 1 & -3+8 i & 5-5 i & -6-4 i \end{pmatrix}^T\]
with a parametrized vector 
\[{\phi_{11}'} = \begin{pmatrix} 1 & -3+8 i & 5-5 i & a+b i \end{pmatrix}^T\]
to obtain a new frame $\Phi'_{a,b}$. This adds parameters to the system of equations \eqref{eq:eqs} by replacing $\ell_{11}$ 
with 
 {\small
\begin{align*} 
\ell_{11}' &= \phi_{11}'^*Q \phi_{11}'= \ (a^2+b^2)x_{44}+2ax_{14}-(6a-16b)x_{24}\\
&+10(a-b)x_{34}-2by_{14}+(16a+6b)y_{24}-10(a+b)y_{34}\\
&+x_{11}-6x_{12}+10x_{13}+73x_{22}-110x_{23}+50x_{33}\\
&-16y_{12}+10y_{13}+50y_{23}.
 \end{align*}
} 
The resulting system of equations has no non-zero real solutions 
for an open subset of $(a,b)\in \R^2$. 
In particular, this includes an open ball around the point $(a,b)=(-6,-4)$.
These points all correspond to frames that define injective measurements on $\C^4$. 

Computing the exact set of $(a,b)\in \R^2$
for which the new frame $\Phi'_{a,b}$ is injective is possible in theory but 
prohibitively time consuming in practice. However, by numerically testing points in a 
$0.1\times 0.1$ grid around the point $(a,b)=(-6,-4)$, we 
can get rough local picture of the open semi-algebraic set of $(a,b)$ for which the 
frame $\Phi'_{a,b}$ defines injective measurements. The result is shown in 
Figure 3.
\pagebreak

\textbf{Acknowledgements.} Thanks to Bernhard Bodmann for his encouragement and
interest in this problem. The author was supported by an NSF postdoc DMS-1204447.

\begin{figure}\begin{center}
\includegraphics[width=2.3in]{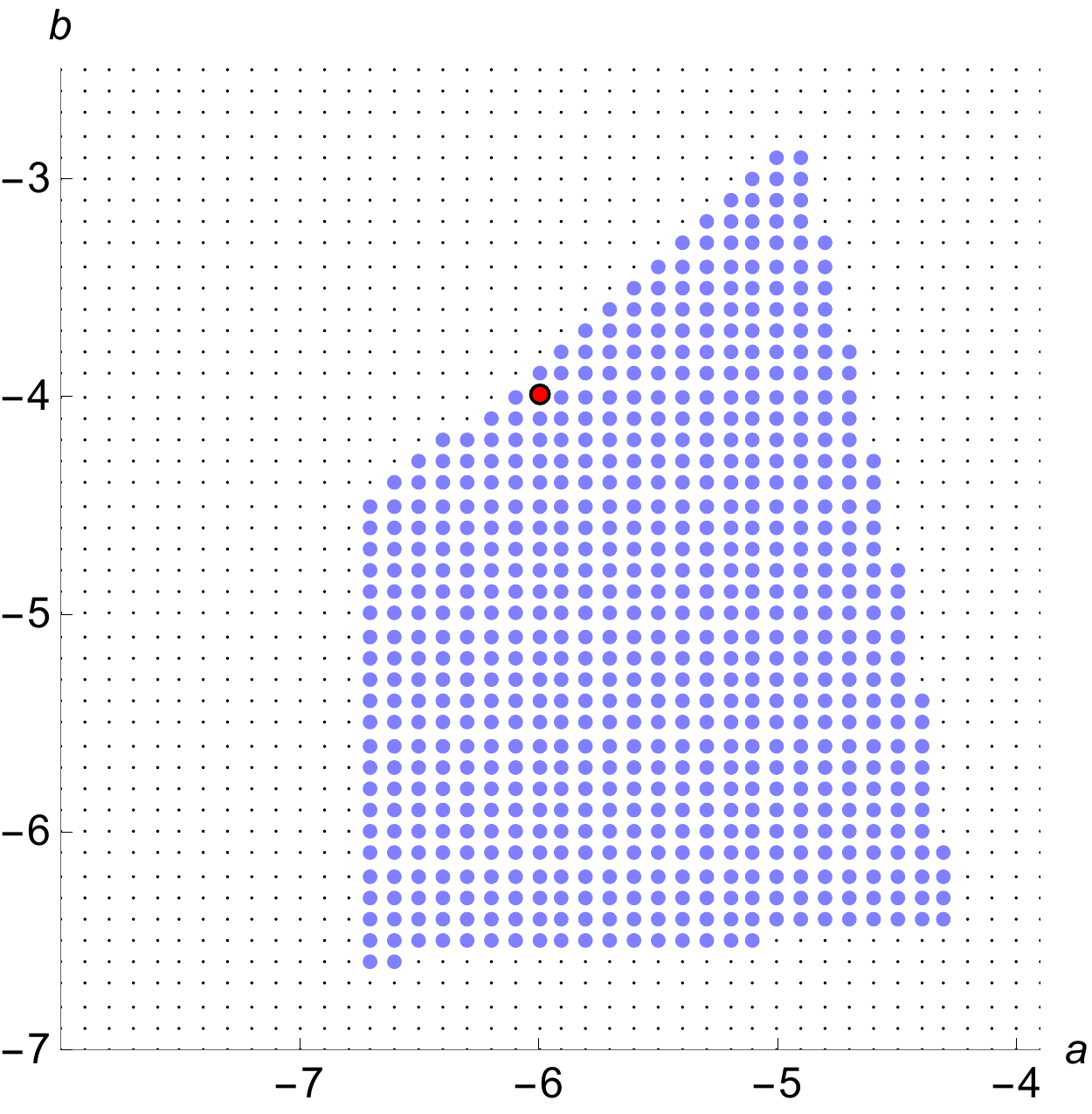}
\end{center} \caption{A sampling of injective frames (blue) around $\Phi$ (red).}
\end{figure}

\bibliographystyle{IEEEtran}
\bibliography{FrameBib}

\end{document}